\documentclass[reqno]{amsart}

\title{Non-formality of \\ the odd dimensional framed little balls operads }

\author{Syunji Moriya}
\thanks{The  author is partially supported by JSPS KAKENHI Grant Number 26800037.}
\address{Department of Mathematics and Information Sciences, Osaka Prefecture University, Sakai, 599-8531, Japan}
\email{moriyasy@gmail.com}
\date{\today}
\usepackage{setspace}
\usepackage{amsthm}
\usepackage{amscd}
\usepackage{amsmath,amsthm,amssymb,amscd,mathrsfs,picture}
\usepackage[arrow,matrix]{xy}
\input xy
\xyoption{all}

\usepackage{enumerate}
\usepackage{graphics}

\theoremstyle{definition}
\newtheorem{defi}{Definition}[section]

\newtheorem{rem}[defi]{Remark}
\theoremstyle{plain}

\newtheorem{lem}[defi]{Lemma}

\newtheorem{thm}[defi]{Theorem}

\newcommand{\C}{\mathcal{C}}% a category
\newcommand{\CH}{\mathcal{CH}}%category of chain

\newcommand{\oper}{\mathcal{O}}% an operad
\newcommand{\A}{\mathcal{A}}% associative operad
\newcommand{\K}{\mathcal{K}}% the Kontsevich operad
\newcommand{\fK}{f\mathcal{K}}% framed Kontsevich operad
\newcommand{\CHS}{\mathcal{S}}% choose two operad
\newcommand{\D}{\mathcal{D}}% little balls operad
\newcommand{\aoper}{\mathcal{P}}% another operad
\newcommand{\homology}{\mathcal{H}}
\newcommand{\CB}{\mathsf{\omega SO}}
\newcommand{\Q}{\mathbb{Q}}

\newcommand{\ttot}{\widetilde{\operatorname{Tot}}}
\newcommand{\R}{\mathbb{R}}

\newenvironment{itemize2}{
 \begin{list}{$\bullet$\ \ }%
 {\setlength{\itemindent}{0pt}
  \setlength{\leftmargin}{1.5em}    % 左のインデント
  \setlength{\rightmargin}{0em}   % 右のインデント
  \setlength{\labelsep}{0.5em}      % 黒丸と説明文の間
  \setlength{\labelwidth}{3em}    % ラベルの幅
  \setlength{\itemsep}{0em}       % 項目ごとの改行幅
  \setlength{\parsep}{0em}        % 段落での改行幅
  \setlength{\listparindent}{0em} % 段落での一字下り
  \setlength{\topsep}{0pt} %　前の行と最初の項目，最後の項目と後の行の改行幅 
}
}{
 \end{list}
}
\allowdisplaybreaks[1]

\begin{document}
\maketitle
\begin{abstract}
We prove that the  chain operad of the framed little balls (or disks)  operad is not formal as a non-symmetric operad over the rationals if the dimension of their balls is odd and greater than 4.
\end{abstract}
\section{Introduction}
A chain operad $\oper$ is said to be formal if  $\oper$ and the homology operad $H_*(\oper)$ is connected by a chain of weak equivalences (see section \ref{Spreliminary} for the precise definition).

The formality of the chain little balls operads is a very important property. It is first discovered by D. Tamarkin in the 2 dimensional case, and applied to the proof of deformation quantization, and generalized by M. Kontsevich to the arbitrary dimension ( see \cite{LV} for a detailed description of Kontsevich's sketchy construction). Combined with Goodwillie-Klein-Weiss embedding calculus, the formality is also used to compute the homology of knot spaces and more generally, embedding spaces  (see \cite{sinha,LTV,ALV,AT,tsopmene2}). \\
 \indent The framed little balls operads are cousins of the little balls operads, which encode rotations of balls, and appear in many areas. For example, They play an important role in the embedding calculus, in particular in the case that involving manifolds are not parallelizable (see \cite{BW, turchin}). \\
\indent J. Giansiracusa and P. Salvatore \cite[Theorem.A]{GS} proved the formality of framed little 2-balls (disks) operad, and they questioned whether the framed little balls operads of higher dimensions are formal.\\
\indent In this paper, we give an answer to this question  for odd dimensions greater than 4.
\begin{thm}\label{Tmain1}
Let $d$ be an odd number greater than 4. If the coefficients are  the rational numbers, the $d$-dimensional chain framed little  balls operad is not formal as a  non-symmetric operad.
\end{thm}
 It is clear that non-formality as a non-symmetric operad  implies non-formality as a  symmetric operad. Theorem \ref{Tmain1} might be an unhappy result in view of computations on embedding spaces but the author  expects the obstruction to formality given in section \ref{Sobstruction} will be useful to understand the higher information of  the chain framed little disks operad. \\
\indent We shall give an outline of the proof of Theorem \ref{Tmain1}. (We actually prove Theorem \ref{TKontsevich}, part 1 of which is equivalent to Theorem \ref{Tmain1}.) We use the framed Kontsevich operad instead of the framed little balls operad since it has a structure of a multiplicative operad. To a multiplicative operad, J. E. McClure and J. H. Smith associated a cosimplical space. We associate a homology spectral sequence to the cosimplicial space by the Bousfield- Kan's procedure.  In section \ref{Snondegeneracy}, We see that the spectral sequences associated to the odd-dimensional framed Kontsevich operads do not degenerate at $E^2$-page. Here we actually deal with  simpler operads, the framed choose-two operads, mainly. The non-degeneracy of the spectral sequences of the framed Kontsevich operads follows from that of the framed choose-two operads. We detect an element whose differential is non zero. The argument is  essentially a simple observation for the number of generators and degrees. Note that formality as a multiplicative operad implies degeneracy of the spectral sequence at $E^2$-page as in \cite{LTV,moriya,tsopmene}, so in turn, non-degeneracy of the spectral sequence implies non-formality as a multiplicative operad. But  it  is weaker than non-formality as a non-symmetric operad.
For example, the 2-dimensional chain Kontsevich operad is not formal as a multiplicative operad though it is formal as (non-)symmetric operad (see \cite{TW}).  To obtain the latter non-formality, we introduce an obstruction to formality which can be defined for any (possibly non-multiplicative) operad the homology of which is isomorphic to that of the framed little balls  operad. This obstruction is something like "Massey triple bracket" for Gerstenhaber bracket (see Remark \ref{Robstruction}). Unlike the usual Massey triple product, our obstruction might depend not only on classes but also on cycles. (At least, we do not prove such independence.) To deal with technical issues about choices of cycles, we make use of a model category of operads.   We prove the main theorem \ref{Tmain1} (or \ref{TKontsevich}) by using the fact that the obstruction is equal to the non-trivial differential of the spectral sequences for the framed operads. \\
\indent Other non-formality results are found in \cite{TW,livernet}. In \cite{TW} the authors proved non-formality of codimension one inclusion between the little balls operads. In \cite{livernet}, the author proved non-formality of Swiss-cheese operad and considered a general Massey product for partial compositions of operads.

\section{Preliminary}\label{Spreliminary}
In this section, we shall recall basic definitions and known results, and  fix notations. 
\begin{itemize}

\item In the present paper, the coefficients of all modules are in the field of  rational numbers $\Q$. In particular, all  homology groups are supposed to have the rational coefficients. 
\item The term  "operad" means non-symmetric operad. A \textit{(non-symmetric) operad} in a symmetric monoidal category $(\C,\otimes,1)$ consists of a sequence of objects $\oper(0)$, $\oper(1),\dots,\oper(n),\dots$ in $\C$ and a set of partial compositions 
$-\circ_i-:\oper(m)\otimes\oper(n)\to \oper(m+n-1)$. which satisfies axioms of associativity and unity (see  \cite[section 2]{muro} or \cite[Variants 1.2]{KM}). A morphism of operads $f:\oper\to\aoper$ is a sequence of morphisms $\{f_n:\oper(n)\to \aoper(n)\}_{n\geq 0}$ which is compatible with the partial compositions and preserves the unit. We call $\oper(n)$ the part of arity $n$ of $\oper$. If $\C$ is the monoidal category  $Top$ of  topological spaces, continuous maps, and the cartesian product (resp.  $\CH$ of chain complex, chain maps, and the tensor product over $\Q$), we call an operad in $\C$ a topological operad (resp. a chain operad). The singular chain functor $C_*:Top\to \CH$ (with the rational coefficients) assigns a chain operad to a topological operad.
\item A morphism $f:\oper\to \aoper$ of topological operads (resp. chain operads)  is said to be a \textit{ weak equivalence} if $f_n:\oper(n)\to \aoper(n)$ is a weak homotopy equivalence (resp. a quasi-isomorphism ) for each $n\geq 0$. A chain operad $\oper$ is said to be \textit{ formal} if there exists a chain of weak equivalences of operads connecting $\oper$ and $H_*(\oper)$:
\[
\xymatrix{\oper &\oper_1\ar[l]\ar[r] &\cdots & \oper_N\ar[l]\ar[r] & H_*(\oper),}
\]
where  $H_*(\oper)$ is a chain operad given by $H_*(\oper)(n)=H_*(\oper(n))$   with the zero differential.
%%%% 9/8
\item We shall deal with the \textit{Kontsevich operad} $\K_d$ and the \textit{choose-two operad $\CHS_d$} (see \cite[section 3, 4, Example 7.4]{sinha}, where the term "choose-two operad" represents a different concept, in their notation, our $\CHS_d$ is equal to $(S^{d-1})^{S^2_\bullet}$, or \cite[section 2]{salvatore}, where the choose-two operad is denoted by $B_n$). We shall recall the definition of operads $\K_d$ and $\CHS_d$ in some details. Let $d$ be a positive integer and   $S^{d-1}=\{x\in\R^d\mid ||x||=1\} $ be the standard unit $(d-1)$-sphere in the Euclidean space $\R^d$. We put 
\[
\CHS_d(n)=\prod_{1\leq i<j\leq n}S^{d-1}.
\]
Let $F_n(\R^d)$ denote the space of ordered configurations of $n$-points in $\R^d$. We define a map $\theta:F_n(\R^d)\to \CHS_d(n)$ by
\[
\theta(x_1,\dots,x_n)=\left(\frac{x_j-x_i}{|x_j-x_i|}\right)_{i,j}
\]
In other words, the $(i,j)$-component of $\theta(x_1,\dots,x_n)$ is the direction vector from $x_i$ to $x_j$. The space $\K_d(n)$ is defined to be the closure of the image of $\theta$ in $\CHS_d(n)$. We set $\CHS_d(n)=\K_d(n)=*$ for $n=0,1$. To get the partial composition of $\CHS_d$, we shall define a map
\[
(-\circ_i-):F_m(\R^d)\times F_n(\R^d)\to \CHS_d(m+n-1)
\] 
for $i=1,\dots, m$. Take $(x_1,\dots,x_m)\in F_m(\R^d)$ and $(y_1,\dots,y_n)\in F_n(\R^d)$. Intuitively speaking, $(x_1,\dots,x_m)\circ_i(y_1,\dots,y_n)$ is represented by the configuration  made from $x_1,\dots,x_m$ by replacing $x_i$ with $y_1,\dots,y_n$ which are infinitesimally rescaled. More precisely, 
we first put 
\[
(w_1^r,\dots, w_{m+n-1}^r)=(x_1,\dots, x_{i-1},x_i+ry_1,\dots x_i+ry_n,x_{i+1},\dots,x_n)
\]
where $ry_k$ means scalar multiplication by a positive number $r$. Note that if $r$ is sufficiently small, $(w_1^r,\dots, w_{m+n-1}^r)$ belongs to $F_{m+n-1}(\R^d)$. Then, we set
\[
(x_1,\dots,x_m)\circ_i(y_1,\dots y_n)=\lim_{r\to 0}
\theta (w_1^r,\dots, w_{m+n-1}^r).
\]
 Note that the limits of direction vectors of $w_1^r,\dots w_{m+n-1}^r$ depend only on the direction vectors between two of  $x_1,\dots, x_m$ and of $y_1,\dots, y_n$, so the above map $(-\circ_i-)$ is naturally extended to a map $(-\circ_i-): \CHS_d(m)\times \CHS_d(n)\longrightarrow \CHS_d(m+n-1)$, which gives $\CHS_d$ a structure of an operad (see \cite[section 2]{salvatore} for an explicit formula of this partial composition).  It is known that the restriction of  this partial composition to $\K_d$ factors through $\K_d$ and  we endow $K_d$ this structure of a sub-operad of $\CHS_d$. (see \cite[Theorem 4.5]{sinha}). $\K_d$ and the  little balls operad $\D_d$ is known to be weak equivalent as  topological operads for each $d\geq 1$. In other words, $\K_d$ and $\D_d$ are connected by a chain of weak equivalences. 
\item The framed version of $\K_d$ and $\CHS_d$ is given as a semidirect product with the rotation group $SO_d$.  A notion of a \textit{semidirect product} of an operad is introduced by N. Wahl and P. Salvatore \cite[Definition 2.1]{SW}. Let $\oper$ be a topological operad and $G$ be a topological group. Suppose each $\oper(n)$ has a $G$-action which satisfies some compatibility conditions (see ibid.). We define a topological operad $\oper\rtimes G$ by
\[
\begin{split}
\oper\rtimes G(n)&=\oper(n)\times G^n\\
(x;g_1,\dots, g_m)\circ_i(y;h_1,\dots,h_n)&=(x\circ_i(g_i\cdot y);g_1,\dots,g_ih_1,\dots, g_ih_n,g_{i+1},\dots, g_m)
\end{split}
\]
for $(x,g_1,\dots,g_m)\in \oper\rtimes G(m)$, $(y,h_1,\dots,h_n)\in \oper\rtimes G(n)$. The most important example of  semi-direct products is the $d$-\textit{dimensional  framed little balls operad}  which is the semi-direct product with respect to the natural action of the rotation group $SO_d$ on the $d$-dimensional little balls operad $\D_d$ (see Example 2.2 of ibid.). $\K_d$ and $\CHS_d$ also have  the actions of $SO_d$ which are induced by the restriction of the natural action on $\R^d$ to $S^{d-1}$. The inclusion $\K_d\to \CHS_d$ preserves these actions. For $\oper =\D_d, \K_d$, or $\CHS_d$, we put $f\oper =\oper\rtimes SO_d$. We call $f\K_d$  (resp. $f\CHS_d$) the $d$-\textit{dimensional framed Kontsevich operad} (resp. the $d$-\textit{dimensional framed choose-two operad}). $f\K_d$ and the framed little balls operad $f\D_d$ are known to be weak equivalent as  topological operads (see \cite[section 3]{salvatore}). 
\item  We will use McClure-Smith's procedure which produces cosimplicial objects $\oper^\bullet$ from a multiplicative operad (see \cite{MS}).  Let $\A$ denote the associative operad. (We consider the unital version so that $\A(0)$ is a point).  A \textit{multiplicative operad} (or an \textit{operad with multiplication}) is a morphism from $\A$ to an operad $\oper$ and a morphism of multiplicative operads is the same as a morphism of the under category. We can associate a cosimplicial space $\oper^\bullet$ to a multiplicative operad $f:\A\to\oper$ as follows. Let $\mu\in\oper(2)$ (resp. $e\in\oper(0)$) be the image of the unique element of $\A(2)$ (resp. of $\A(0)$) by $f$. We put $\oper^n=\oper(n)$ for each integer $n\geq 0$, and we define the coface $d^i:\oper^n\to\oper^{n+1}$ and codegeneracy $s^i:\oper^n\to\oper^{n-1}$ by 
\[
\begin{split}
d^0(x)=\mu\circ_2 x,\ d^{n+1}(x)&=\mu\circ_1 x,\ d^i(x)=x\circ_i\mu\ \quad (1\leq i\leq n)\\
s^i(x)&=x\circ_i e \quad (1\leq i\leq n)
\end{split}
\]
The main advantage of the Kontsevich operad to the little balls operad is that it has a structure of mulitiplicative operad. We set $v_0=(1,0,\dots,0)$ as the base point of $S^{d-1}$. The map $\A(2)\to \CHS_d(2)=S^{d-1}$ taking the value on $v_0$ uniquely extends to a morphism $\A\to\CHS_d$. This morphism factors through $\K_d$. We regard $\CHS_d$ and $\K_d$ as multiplicative operads with these morphisms 
%so   the inclusion $\K_d\to\CHS_d$ is compatible 
(see \cite[section2]{salvatore}, \ \cite[Proposition 4.7]{sinha} ). For $\oper=\CHS_d,\ \K_d$, We also regard $f\oper$ as a multiplicative operad with the composition $\A\to\oper\subset f\oper$ where the inclusion is induced by the inclusion to the unit $*\to SO(d)$. We deal with  the cosimplicial objects $\CHS_d^\bullet$, $\K_d^\bullet$, $f\CHS_d^\bullet$, and $f\K_d^\bullet$.    
\item For a cosimplicial space $Y^\bullet$   let  $E^r_{p,q}(Y^\bullet )$ denote the $E^r$-page of Bousfield-Kan spectral sequence associated to the cosimplicial chain complex 
$C_*(Y^\bullet)$ (Note that $E^r_{p,q}$ is the part of cosimplicial degree $-p$, chain degree $q$). 
\item  For a topological multiplicative operad $\oper$, $HH_{*,*}(H_*(\oper))$ denotes the Hochschild homology of $H_*(\oper)$ considered as a chain operad with the zero differential. It is the  homology group of the normalization of the cosimplical graded  group $H_*(\oper)$ (see \cite[section 3]{T},\cite[Definition 18]{salvatore}, \cite[Definition 2.17]{sinha}). The Hochschild homology  has the  bigrading analogous to that of $E^r_{*,*}(Y^\bullet)$ and a natural structure of a Gerstenhaber alegbra. (see \cite[section 2]{GV} where a Gerstenhaber algebra is called a G-algebra , and \cite[section 3]{T},\cite[subsection 4.3]{sakai}).
$E^2_{*,*}(\oper^\bullet )$ also has a  structure of a Gerstenhaber algebra as it is naturally isomorphic to $HH_{*,*}(H_*(\oper))$

\item For $k=d,\ d-1$, $\CB_k$ denotes  the cobar complex for the coalgebra $H_*(SO_k)$ with Alexander-Whitney diagonal. $H_{*,*}(\CB_k)$ denotes the total homology with the natural bigrading. 
\item For a pointed topological space $X$, let $\underline{X}^\bullet$ denote the usual cobar cosimplicial space (so $\underline{X}^n=X^{\times n}$), and $\Omega(X)$ denote the based loop space of $X$. We let  $\ttot$ denote the homotopy totalization (or limit) functor for  cosimplicial spaces.
\end{itemize}
\begin{lem}[Corollary 10 and Proposition 16 of  \cite{salvatore}]\label{Ltot} For any integer $d\geq 2$, we have weak homotopy equivalences:
\[
\ttot(\CHS_d^\bullet)\simeq \Omega^2 (S^{d-1}), \quad   
\ttot(f\CHS_d^{\bullet})\simeq \Omega(SO_{d-1}).
\] The second equivalence is induced by the cosimplicial map 
$\underline{SO_{d-1}}^\bullet \to f\CHS_d^\bullet $ defined  by the usual inclusion $SO_{d-1}\to SO_d$ to the subgroup consisting of matrices whose first column is $v_0$, and  the configurations which are the images by the fixed morphism $\A\to f\CHS_d$.
\begin{flushright}
\qedsymbol
\end{flushright}
\end{lem}

%%%%%%%%%%%%%%%%%%%%%%%%%%%%5/15
\indent The following lemma is  well-known. We denote by $\Omega(SO_k)_1$ the component of the based loop space of $\Omega(SO_k)$ containing constant loops for  $k=d,\  d-1$.
\begin{lem}\label{Lrotation} Let $d=2m+1$ be an odd number greater than 2.\\
(1) There exist isomorphisms of rational homology algebras:
\[
H_*(SO_d)\cong \bigwedge (\beta_1,\dots,\beta_m), \quad H_*(SO_{d-1})\cong \bigwedge(\beta_1,\dots,\beta_{m-1},e)\] Here, $\bigwedge$ denotes the free anti-commutative algebra, and we set $\deg\beta_i=4i-1, \ \deg e=2m-1$. Furthermore, the generators can be taken as primitive elements with respect to the Alexander-Whitney diagonal.\\
(2) There exist isomorphisms of rational homology algebras:
\[
H_*(\Omega(SO_d)_1)\cong \Q [\gamma_1,\dots,\gamma_m],\quad   H_*(\Omega (SO_{d-1})_1)\cong \Q[\gamma_1,\dots,\gamma_{m-1},f]
\] Here, $\Q[\cdots]$ denotes the free commutative algebra, and we set $\deg\gamma_i=4i-2, \ \deg f=2m-2$.
\begin{flushright}
\qedsymbol
\end{flushright}
\end{lem}
As the framed little balls operads are weak equivalent to the framed Kontsevich operads, Theorem \ref{Tmain1} is equivalent to the part 1 of the following theorem  which we will prove in the rest of the paper.
\begin{thm}\label{TKontsevich}
Let $d$ be an  odd number greater than 4.
\begin{itemize2}
\item[(1)] The chain operad $C_*(f\K_d)$ of the framed $d$-dimensional Kontsevich operad is not formal over $\Q$.
\item[(2)] The chain operad $C_*(f\CHS_d)$ of the framed $d$-dimensional choose-two operad is not formal over $\Q$.
\end{itemize2}

\end{thm}
\section{Non-degeneracy of the spectral sequences associated to the framed operads}\label{Snondegeneracy}

\indent In the rest of paper, $d=2m+1$ denotes an odd number greater than 4.  
\begin{lem}\label{Lspectralseq}
Let $\oper$ be $\K_d$ or $\CHS_d$.
\begin{itemize2}
\item[(1)] 
There exists an isomorphism of algebras
\begin{equation}\label{EQspectral}
E^2_{p,q}(f\oper^\bullet)\cong \bigoplus_{p_1+p_2=p, q_1+q_2=q}
HH_{p_1,q_1}(H_*(\oper))\otimes H_{p_2,q_2}(\CB_{d})
\end{equation}
which is natural with respect to the inclusion $i:\K_d\to \CHS_d$, and 
\item[(2)] Let $\bar\beta_m$ be the image of the element $\beta_m$ in Lemma \ref{Lrotation} by the natural map $H_{4m-1}(SO_d)\to H_{-1,4m-1}(\CB_{d})$. When 
we regard $1\otimes \bar\beta_m$ as an element of $E^2(f\oper^\bullet)$ under the isomorphism of (1), $d^2(1\otimes \bar\beta_m)$   is non-zero in $E^2(f\oper^\bullet)$. 
\end{itemize2}
\end{lem}
\begin{proof}
We shall prove the part 1. There is a natural isomorphism $E^2_{*,*}(f\oper^\bullet)\cong HH_{*,*}(H_*(f\oper))$. The action of the Hopf algebra $H_*(SO_d)$ on $H_*(\oper (n))$ is trivial by the degree reason for any odd $d$, see \cite[Theorem 5.4]{SW}, so the semidirect product splits on the homology level (this is the point where we need $d$ is odd). It follows that as a cosimplical graded vector space , $H_*(f\oper^\bullet)$ is isomorphic to the  cosimplicial-levelwise tensor product of $H_*(\oper^\bullet)$ and the cosimplicial cobar complex associated to $H_*(SO_d)$. This  implies $HH_{*,*}(H_*(f\oper))\cong HH_{*,*}(H_*(\oper))\otimes H_{*,*}(\CB_d)$.\\
\indent  We shall show the part 2 for $\oper=\CHS_d$. In view of part 1, we easily see  $E^2_{p,q}(f\CHS_d^\bullet)=0$ for $q<-\frac{d-1}{2}p$. As $d\geq 5$, we have the inequality $\frac{d-1}{2}>1$. Also, note that $H_*(\Omega((SO_{d-1})_\Q))\cong H_*(\Omega(SO_{d-1})_1)$ as $\pi_1(SO_{d-1})\otimes \Q=0$, where $(-)_\Q$ denotes the rationalization. These facts and \cite[Theorem 3.4]{bousfield}  imply $E^r_{*,*}(f\CHS_d^\bullet)$  converges  to  $H_*(\Omega(SO_{d-1})_1)$. Note that the partial compositions of $\CHS_d$ is defined by the diagonal map of $S^{d-1}$.  This fact and the formality of the rational singular chain coalgebra $C_*(S^{d-1})$ imply the formality of the rational singular chain multiplicative operad $C_*(\CHS_d)$ . This multiplicative operad formality implies the Bousfield-Kan  spectral sequence associated to the cosimplicial space $\CHS_d^\bullet$  degenerates at $E^2$-page, see the proof of Theorem 1.4 in \cite{moriya}. This spectral sequence converges to the homology $H_*(\ttot(\CHS_d^\bullet))$. By these observations and  Lemma \ref{Ltot}, we have  an isomorphism
\begin{equation}\label{EQhochschild}
HH_{*,*}(H_*(\CHS_d))\cong H_{*}(\Omega^2S^{d-1}).
\end{equation}
Moreover,  this isomorphism is an isomorphism of Gerstenhaber algebras.  
On the other hand, the following isomorphism of graded algebras is well-known. 
\begin{equation}\label{EQcobar}
H_{*,*}(\CB_d)\cong H_*(\Omega(SO_d)_1)
\end{equation}
Putting the isomorphisms (\ref{EQspectral}), (\ref{EQhochschild}) and (\ref{EQcobar}) into together, we obtain the following:
\begin{equation}\label{EQframedchoosetwospec}
E^2(f\CHS_d^\bullet)\cong H_*(\Omega^2S^{d-1})\otimes H_*(\Omega(SO_d)_1)\Longrightarrow H_*(\Omega(SO_{d-1})_1)
\end{equation}
Let $\{-,-\}$ denote the Gerstenhaber bracket (Brawder operation) on $H_*(\Omega^2S^{d-1})$. It is well-known that $H_*(\Omega^2S^{d-1})$ is generated by two elements $x, \{x,x\}$ with $\deg x=d-3$ as a graded commutative algebra when $d$ is odd. So with Lemma \ref{Lrotation}, the graded algebras in (\ref{EQframedchoosetwospec}) have the following generators
\begin{equation}\label{EQgenerators}
\begin{split}
H_*(\Omega^2 S^{d-1})\otimes H_*(\Omega (SO_d)_1)\ : \ &x\otimes 1,\ \{ x,x\}\otimes 1,\ 1\otimes \gamma_1,\dots, 1\otimes\gamma_{m-1},\ 1\otimes\gamma _m \\
H_*(\Omega(SO_{d-1})_1)\ : \ & \ \ f\ ,\phantom{\{x,x\}\otimes 1 xxx}\ \gamma_1,\quad \dots,\quad \gamma_{m-1}
\end{split}
\end{equation}
In the following we consider the elements in the  upper horizontal line of (\ref{EQgenerators}) as elements in $E^2(f\CHS_d^\bullet)$. By comparison of the number of generators and their degree, it is plausible to expect the equation
\begin{equation}\label{EQdifferentialgamma}
d^2(1\otimes\gamma_m)=\{x,x\}\quad  \text{ up to non-zero scaler multiple.}
\end{equation}  
We shall verify this equation (\ref{EQdifferentialgamma}) in details.
Note that there exist morphisms of spectral sequences as follows.
%\begin{equation}\label{EQframedchoosetwo}
%\CHS_d^\bullet \longrightarrow f\CHS_d^\bullet\longleftarrow \underline{SO_{d-1}}^\bullet.
%\end{equation}
 %These morphisms induce morphisms of spectral sequences
\begin{equation}\label{EQspectralchoosetwo}
E^r_{p,q}(\CHS_d^\bullet)\longrightarrow E^r_{p,q}(f\CHS_d^\bullet)\longleftarrow E^r_{p,q}(\underline{SO_{d-1}}^\bullet).
\end{equation}
Here the left morphism is induced by the inclusion to the unity $*\to SO_d$ and the right one is induced by  the morphism in Lemma \ref{Ltot}.
We shall consider the case $r=2$. Note that $x\otimes 1$, $\{x,x\}\otimes 1$ come from the left hand side of the diagram (\ref{EQspectralchoosetwo}) and $1\otimes\gamma_i$ ($1\leq i\leq m-1$) comes from the right hand side of the same diagram in view of the isomorphism (\ref{EQhochschild}) and the isomorphism $E^2(\underline{SO_{d-1}}^\bullet)\cong H_*(\CB_{d-1})\cong H_*(\Omega(SO_{d-1})_1)$, see also Lemma \ref{Lrotation}. This means these elements are  cycles at any page of the middle spectral sequence in (\ref{EQspectralchoosetwo}) as the left and right hand sides of (\ref{EQspectralchoosetwo}) degenerate at $E^2$-page.  $\{x,x\}$ has the odd  degree and the all generators of the abutment $H_*(\Omega(SO_{d-1})_1)$ have even degree so $\{x,x\}\otimes 1$ must be a boundary in some page.
Suppose $\{x,x\}\otimes 1$ is  a boundary in $E^{r}$-page but not in $E^{r-1}$-page.   Take an element $y$ such that $d^r(y)=\{x,x\}\otimes 1$ We may write $y=y_1+k(1\otimes\gamma_m)$ where $y_1$ is a polynomial of $x\otimes 1$, $1\otimes\gamma_1,\dots, 1\otimes \gamma_{m-1}$, and $k$ is a scalar. But $y_1$ is a cycle since $1\otimes\gamma_1,\dots,1\otimes\gamma_{m-1}$ are cycles. This implies $d^r(y_1)=0$ and in turn, we have $kd^r(1\otimes\bar\gamma_m)=\{x,x\}\otimes 1$. %as the differentials are derivative for the multiplications, we conclude that the following equation holds up to scalar multiplication:
Note that  $1\otimes
 \gamma_m$ corresponds to $1\otimes \bar\beta_m\in 1\otimes H_{-1,4m-1}(\CB_d)$ under the isomorphism (\ref{EQcobar}) and  $x$ corresponds an element $\alpha \in H_{d-1}(\CHS_d(2))\subset HH_{-2,d-1}(H_*(\CHS_d))$ and $\{x,x\}$ corresponds $\{\alpha,\alpha\}\in HH_{-3,2d-2}(H_*(\CHS_d(3))$ under the  isomorphism (\ref{EQhochschild}), where the bracket $\{\alpha,\alpha\}$ denotes the algebraic Gerstenhaber bracket recalled in  section \ref{Spreliminary} so  the bidegree of $d^r$ must be $(-2,1)$, which implies  $r=2$, and we have  proved the equation (\ref{EQdifferentialgamma}) and the part 2 of the lemma for the  case $\oper=\CHS_d$.\\
\indent For the case $\oper=\K_d$,  As the spectral sequences are natural for the inclusion $\K_d\subset \CHS_d$ and the image of $\{x,x\}$ by the inclusion is non-zero, which can be seen by elementary computation based on \cite[Theorem 7.4]{sinha},  the above result for $\oper=\CHS_d$ immediately  implies  $d^2(1\otimes \bar\beta_m)\not=0$ for $\oper=\K_d$. 
\end{proof}

%%%%%%%%%%%%%%%%%%%%%%%%%%%%%%%%%%%%%%%%%%%%%

\section{An obstruction to formality and proof of Theorem \ref{TKontsevich}}\label{Sobstruction}
In Lemma \ref{Lrotation} , we may take  a  generator $\beta_m\in H_{4m-1}(SO_d)$ to be primitive. In other words, we may assume 
\begin{equation}\label{EQprimitive}
\Delta\beta_m=1\otimes \beta_m+\beta_m\otimes 1
\end{equation} where $\Delta$ denotes the Alexander-Whitney diagonal. We regard $\beta_m$ as an element of $H_*(\fK_d(1))$ by the obvious homeomorphism $\fK_d(1)=\{id\}\times SO_d=SO_d$. By the definition of a semi-direct product (see section \ref{Spreliminary}) we  see that the following diagram is commutative.
\[
\xymatrix{SO_d\ar[r]^{\Delta\qquad\qquad }\ar@{=}[d] & SO_d\times SO_d\times SO_d\ar@{=}[r]&SO_d\times f\K_d(1)\times f\K_d(1)\ar[d]^{\psi}\\
f\K_d(1)\ar[rr]^{\varphi}&&f\K_d(2)}
\]
Here, $\Delta$ denotes the (topological) diagonal map, and  the maps $\psi$ and $\varphi$ are given by $\psi(x,y,z)=(x\cdot \mu)\circ (y,z)$ and $\varphi(y)=y\circ_1 \mu$ where $\mu$ is   the image of the unique point by the structure map $\A(2)\to f\K_d(2)$.  We also have an equation $\beta_m\cdot\mu=0$ as $H_{4m-1}(\K_d(2))=0$. By this equation, the equation (\ref{EQprimitive}), and the above diagram, we have the following equation.
\begin{equation}\label{EQconditionbeta}
\beta_m\circ_1\mu=\mu\circ_2\beta_m+\mu\circ_1\beta_m\in H_{4m-1}(f\K_d(2)).
\end{equation}
 In view of this one, we define an obstruction class for formality as follows. 
\begin{defi}\label{Dobstruction}
Let $\oper$ be a chain operad such that the homology $H_*(\oper)$ is isomorphic to $H_*(\fK_d)$ as a graded operad. Let $\nu\in \oper(2)_0$ be a cycle which represents a generater of $H_0(\oper(2))\cong \Q$. Let $g\in \oper(1)_{4m-1}$ be a cycle.  We impose the pair $(\nu, g)$ the following condition:  
\begin{equation}\label{EQconditiongnu}
[g\circ_1 \nu ]=[\nu\circ_2 g+\nu\circ_1 g] \in H_{4m-1}(\oper(2)).
\end{equation}
Here, $[a]$ denotes the homology class represented by a cycle $a$. \\ 
\indent We can pick an element $h \in \oper(2)_{4m}$ such that 
\begin{equation}\label{EQconditionh}
dh=\nu\circ_2g+\nu\circ_1g-g\circ_1\nu .
\end{equation} As $[\nu]$ represents an associative multiplication, we can pick an element $\xi\in \oper(3)_1$ such that
\begin{equation}\label{EQconditionxi}
d\xi=\nu\circ_2\nu-\nu\circ_1\nu.
\end{equation}
Then, we define an \textit{obstruction cycle}  $\omega=\omega_{\oper}(\nu,g)\in\oper(3)_{4m}$ by
\begin{equation}\label{EQomega}
\begin{split}
\omega=\omega_1-\omega_2,\quad &\omega_1=\nu\circ_2h-h\circ_1\nu+h\circ_2\nu-\nu\circ_1h \\
& \omega_2=g\circ_1 \xi+\xi\circ_1 g+\xi\circ_2 g+\xi\circ_3 g
\end{split}
\end{equation}
We will see this is a cycle for the internal differential of $\oper(3)$ in the proof of Lemma \ref{Lobstruction}.

As $[\nu]$ defines a structure of a multiplicative operad on $H_*(\oper)$, we can consider the Hochschild complex of $H_*(\oper)$. Its differential $\delta_\nu:H_*(\oper(n))\to H_*(\oper(n+1))$ is given by
$\delta_{\nu}([x])=[\nu\circ_2x+\sum_{i=1}^{n-1}(-1)^nx\circ_i\nu+(-1)^n\nu\circ_1x]$\\
\indent We call the element $[\omega]$ of the quotient $H_{4m}(\oper(3))/ \delta_\nu H_{4m}(\oper(2))$ represented by $\omega$ the 
\textit{obstruction class (or set)} for the pair $(\nu,g)$.
\begin{flushright}
\qedsymbol
\end{flushright}
\end{defi}
\begin{lem}\label{Lobstruction}
Under the  notations of Definition \ref{Dobstruction}, $\omega$ is a cycle and the corresponding class $[\omega]$ is independent of choise of $h$ and $\xi$. 
\end{lem}
\begin{proof}
We first show $\omega$ is a cycle. 
\[
\begin{split}
d\omega_1&=\nu\circ_2dh-dh\circ_1\nu+ dh\circ_2\nu -\nu\circ_1dh \\
&=\nu\circ_2(\nu\circ_2g+\nu\circ_1g-g\circ _1\nu)-(\nu\circ_2g+\nu\circ_1g-g\circ_1 \nu)\circ_1\nu\\
&+(\nu\circ_2g+\nu\circ_1g-g\circ_1 \nu)\circ_2\nu-\nu\circ_1(\nu\circ_2g+\nu\circ_1g-g\circ_1 \nu)
\end{split}
\]
By the associativity of partial composition,  we have $\nu\circ_2(\nu\circ_2g)=(\nu\circ_2\nu)\circ_3g$ and $(\nu\circ_2g)\circ_1\nu=(\nu\circ_1\nu)\circ_3g$, for example. By using these and similar equalities, we have
\[
\begin{split}
d\omega_1 &=-g\circ_1 (\nu\circ_2\nu-\nu\circ_1\nu)+(\nu\circ_2\nu-\nu\circ_1\nu)\circ_1g\\
&
+(\nu\circ_2\nu-\nu\circ_1\nu)\circ_2g+(\nu\circ_2\nu-\nu\circ_1\nu)\circ_3g\\
&=-g\circ_1 d\xi+ (d\xi)\circ_1g+(d\xi)\circ_2g+(d\xi)\circ_3g\\
&=d\omega_2
\end{split}
\]
Thus, we have $d\omega=d(\omega_1-\omega_2)=0$. \\
\indent Let $h'$, $h''$ (resp. $\xi'$, $\xi''$) be two elements satisfying the condition (\ref{EQconditionh}) (resp. (\ref{EQconditionxi})) in Definition \ref{Dobstruction}. In the rest of the proof,  $\omega'=\omega'_1-\omega'_2$ and  $\omega''=\omega''_1-\omega''_2$ denote the elements defined by the equation (\ref{EQomega}),
 using $(h',\xi')$ and $(h'',\xi'')$ respectively.    As $ h'-h''$ is a cycle, the class of 
$\omega'_1-\omega''_1$
%is the image of $[\bar h]$ by $\delta_{\nu}$ and  
belongs to $\delta_\nu H_{4m}(\oper(2))$. As $H_1(\oper(3))\cong H_1(f\K_d(3))=0$ and $\xi'-\xi''\in \oper(3)_1$ is a cycle, it is a boundary. So $\omega'_2-\omega''_2$ is also boundary. Thus $\omega'-\omega''$ represents the zero class  in $H_{4m}(\oper(3))/ \delta_\nu H_{4m}(\oper(2))$. In other words, $[\omega']=[\omega'']$.
\end{proof}
%%%%%% 3/13 kokomade
The class $[\omega]$ may depend  not only on the \textit{classes} $[\nu]$ and $[g]$ but also on the \textit{cycles}  $\nu$ and $g$. We must take care about these choices in the  proof of Theorem \ref{TKontsevich}. 
\begin{rem}\label{Robstruction}
Note that the equation (\ref{EQconditiongnu}) is equivalent to the equation $\{[g],[\nu]\}=0$, where $\{-,-\}$ denotes the Gerstenhaber bracket on the Hochschild homology of $H_*(\oper)$.
and the associativity equation $ [\nu]\circ_2[\nu]-[\nu]\circ_1[\nu]=0$ is equivalent to the equation $\frac{1}{2}\{[\nu],[\nu]\}=0$. Our class $[\omega]$ is something like "Massey triple bracket" for $[g], [\nu],[\nu]$. This point of view was pointed out to the author by V. Turchin. The cycle $\omega$ should be defined as $\{\nu, h\}\pm\{g,\xi\}$ under the notations of Definition \ref{Dobstruction} if the Hochschild complex of $\oper$ were a differential graded Lie algebra (with some degree shift). But actually it is not so because $\{-,-\}$ is not (anti-)derivative for the internal differential of $\oper$, and the  actual definition is different from the above formula in signs. This makes it difficult to prove the class $[\omega]$ does not depend on choices of cycles $g$, $\nu$ in each class. 
\end{rem}
We use a model category of chain operads. For general theory of model categories, see \cite{hovey}. It is known that the category of chain operads has a model category structure where weak equivalences are the same as those given in section \ref{Spreliminary}  (see \cite[Theorem 2.1]{moriya} or \cite[Theorem 1.1]{muro}, see also \cite{hinich} for a model category of symmetric operads).
\begin{proof}[Proof of Theorem \ref{TKontsevich}]
We shall prove part 1. The proof of part 2 is completely analogous and we omit it.\\
\indent Set $\oper=C_*(f\K_d)$. Let $\nu\in \oper(2)$ be the $0$-cycle  represented by the image of the unique point  by the structure morphism $\A\to f\K_d$ and $g\in \oper(1)_{4m-1}$ be any cycle which represents the primitive generator $\beta_m$. The pair $(\nu, g)$ satisfies the condition (\ref{EQconditiongnu}) in Definition \ref{Dobstruction}. As $\nu$ is strictly associative, we may take zero  as $\xi$. In this case, by definition, $E^2(f\oper^\bullet)$ is naturally considered as a sub-vector space of $H_*(\oper(-))/\delta_{\nu}H_*(\oper(-))$.  Under this identification, we easily see $[\omega(\nu, g)]=d_2(1\otimes \bar \beta_{4m})$  by unwinding the definition of the differential of the spectral sequence. So by the part 2 of Lemma \ref{Lspectralseq}, we see $[\omega] \in H_{4m}(\oper(3))/ \delta_\nu H_{4m}(\oper(2))$ is non-zero for this choise of $\nu$ and $g$.
Let $\A_\infty$ be the Stasheff's associahedral chain operad. We consider the non-unital version i.e., $\A_\infty(0)=0$ so $\A_\infty$ is a cofibrant operad and has a set of generators $\{\nu_i\mid i\geq 1\}$ consisting of the fundamental classes of cells of the associahedra (see \cite{stasheff,MT}). We define a morphism of operads $f:\A_\infty\to \oper$ by $\nu_1\mapsto \nu$, and taking the other generators to  zeros.  By a functorial factorization, $f$ is factorized as $\xymatrix{\A_\infty\ \ar@{>->}[r]_i&\aoper\ar@{->>}[r]^{\sim}_{p}&\oper}$. As $\A_\infty$ is cofibrant, so is $\aoper$. \\
\indent Set $\homology =H_*(\oper)$. Suppose $\oper$ is formal. In other words, $\oper$ and $\homology$ are connected by a chain of weak equialences of operads. As $\aoper$ is cofibrant (and any chain operad is fibrant), by the theory of model categories, there exists a weak equivalence $q: \aoper \to \homology$. We can take a cycle $g'\in\aoper(1)$  such that the class $[g']$ goes to $\beta_m$ by $p_*$.  The pair $(i(\nu_1),g')$ satisfies the condition of Definition \ref{Dobstruction} for $\aoper$.  It is clear that the pairs $(\nu, p(g'))$, $(q(g'), qi(\nu_1))$ also  satisfy the condition  in Definition \ref{Dobstruction}. We have isomorphisms
\begin{equation}
\xymatrix@R=0pt@M=9pt{
H_*(\oper) &H_*(\aoper)\ar[l]_{p_*}^{\cong}\ar[r]^{q_* }_{\cong }&H_*(\homology)\\
\rotatebox{90}{$\in$}&\rotatebox{90}{$\in$}&\rotatebox{90}{$\in$}\\
\omega(\nu, p(g')) &\omega(i(\nu_1),g')\ar@{|->}[l]\ar@{|->}[r]
&
\omega(qi(\nu_1),q(g'))
}
\end{equation}
By these isomorphisms, $[\omega(\nu, p(g'))]$ corresponds to $[\omega(qi(\nu_1), q(g'))]$. As we show in the above, $[\omega(\nu, p(g'))]$ is non-zero. On the other hand, the differential of $\homology$ is zero, we may choose zeros  as $h$ and $\xi$ in the definition of $\omega(qi(\nu_1),q(g'))$.  Hence $\omega(qi(\nu_1), q(g'))$ represents zero in $H_{4m}(\homology(3))/ \delta_{q(\nu_1)} H_{4m}(\homology(2))$. This is a contradiction.
\end{proof}
\section*{Acknowledgements}
The author is grateful to Thomas Goodwillie, Robin Koytcheff, Dev Prakash Sinha, Paul Arnaud Songhafouo Tsopm\'en\'e, and Victor Turchin for very valuable discussions and comments about the subject of this paper. He deeply thanks Dev Sinha and Ismar Voli\'c for  giving me an opportunity to meet the above people and taking care of my trip kindly. He is also grateful to Masana Harada for fruitful comments for this paper and refreshing conversation.

\end{document}